\documentclass{amsart}

\usepackage{amsmath,amssymb,bbm}
\usepackage{amsfonts}
\usepackage{MnSymbol}
\usepackage{mathrsfs}

\usepackage{enumerate}
\usepackage[notref,notcite]{showkeys}
\usepackage{hyperref}

\theoremstyle{plain} 

\newtheorem{thm}{Theorem}[section]

\newtheorem{lem}[thm]{Lemma}
\newtheorem{prop}[thm]{Proposition}
\newtheorem{exam}[thm]{Example}

\newtheorem{defn}[thm]{Definition}

\theoremstyle{remark}

\def\la{{\langle}}
\def\ra{{\rangle}}
\def\f{\frac}

\def\({\left(}
\def \){ \right)}
\def\[{\left[}
\def \]{ \right]}

 \def\u{{\mathbf{u}}}
 
 \def\x{{\mathbf{x}}}

 \def\d{{\textnormal{d}}}
 \def\e{{\textnormal{e}}}
 \def\i{{\textnormal{i}}}

 \def\a{{\alpha}}
 \def\b{{\beta}}
 \def\g{{\gamma}}
 
 \def\t{{\theta}}
 \def\l{{\lambda}}

 \def\s{{\sigma}}
 
 \def\ub{{\mathbf u}}
 
 \def\xb{{\mathbf x}}

 \def\CI{{\mathcal I}}

 \def\CV{{\mathcal V}}
 
 \def\CW{{\mathcal W}}

 \def\NN{{\mathbb N}}
 \def\PP{{\mathbb P}}
 
 \def\RR{{\mathbb R}}

 \def\sC{{\mathsf C}}

 \def\fs{{\mathfrak s}}
 
  \def\dim{\operatorname{dim}}

  \def\rank{\operatorname{rank}}

\def\Tr{\mathsf{T}}

\newif\ifpdf
\ifx\pdfoutput\undefined
  \pdffalse
\else
  \pdftrue
\fi

\ifpdf
  \usepackage[pdftex]{graphicx}
  \DeclareGraphicsExtensions{.pdf,.jpg,.png}
\else
  \usepackage{graphicx}
\fi
\graphicspath{{Figures/}}

\graphicspath{{./}}

\begin{document}

\title{Minimal cubature rules and Koornwinder polynomials}

\author{Yuan Xu}
\address{Department of Mathematics\\ University of Oregon\\
    Eugene, Oregon 97403-1222.}\email{yuan@uoregon.edu}

\dedicatory{Dedicate to Tom Koornwinder for his pioneer work}

\date{\today}
\thanks{The work was partly supported by Simons Foundation Grant \#849676.}

\begin{abstract}
In his classical paper \cite{K74a}, Koornwinder studied a family of orthogonal polynomials of two variables, 
derived from symmetric polynomials. This family possesses a rare property that orthogonal polynomials of 
degree $n$ have $n(n+1)/2$ real common zeros, which leads to important examples in the theory of 
minimal cubature rules. This paper aims to give an account of the minimal cubature rules of two variables 
and examples originating from Koornwinder polynomials, and we will also provide further examples.  
\end{abstract}

\maketitle

\section{Introduction}
\setcounter{equation}{0}

The two families of orthogonal polynomials of two variables introduced by Koornwinder in \cite{K74a, K74b}, 
derived via symmetric polynomials, are important for minimal cubature rules. We discuss the impact of the first
family in this context. 

A cubature rule is a finite linear combination of point evaluations that gives an approximation to a multivariate integral. 
It is called a quadrature rule if the integral is of one variable. We are interested in the cubature rules of two variables. 
Let $\Omega$ be a domain in $\RR^2$ and let $\CI(f)$ be an integral on $\Omega$, 
$$
       \CI f = \int_\Omega f(\x) W(\x) \d \x, \qquad \x = (x_1,x_2), 
$$
where $W$ is a fixed nonnegative function on $\Omega$. Let $m$ be a positive integer. A cubature rule, $\sC(f)$, 
of degree of precision $m$ is given by 
$$
  \sC (f) = \sum_{k=1}^N \l_k f(\xb_k), \quad \hbox{such that} \quad \CI f = \sC(f), \quad \forall f \in \Pi_m^2,
$$
where $\Pi_m^2$ denote the space of algebraic polynomials of degree $m$ in two variables. The points $\x_k$ are 
called nodes and $\l_k$ weights of the cubature rule. Cubature rules are essential for discretizing integrals and provide 
accurate numerical approximation to integrals. The most useful one has a high degree of precision, equivalently, 
has fewer nodes for a fixed degree of precision. A cubature rule is called {\it minimal} if the number of nodes 
is the smallest among all cubature rules of the same degree for $\CI f$. 

For integral in one variable, the Gauss quadrature rule of degree $2n-1$ is minimal with $n$ nodes and exists
for every integral $\CI f$ on an interval. Moreover, its nodes are zeros of the orthogonal polynomial of degree $n$. 
The same, however, does not hold for cubature rules. It is known that the number of nodes, $N$, of a cubature rule 
of degree $2n-1$ satisfies a lower bound (\cite{My, St})
\begin{equation} \label{eq:old_lwbd}
      N \ge \dim \Pi_{n-1}^2 = \binom{n+1}{2}. 
\end{equation}
We call a cubature rule that attains the lower bound Gauss cubature rule. It is known that a Gauss cubature 
rule of degree $2n-1$ exists if, and only if, its nodes are common zeros of all orthogonal polynomials of degree $n$
in two variables, which means that $n+1$ polynomial curves of degree $n$ intersects at $\frac12 n(n+1)$ real points. 
Thus, it is not surprising that Gause cubature rules exist rarely, They do not exist, for example, if $W$ and $\Omega$ 
are symmetric with respect to the origin, called centrally symmetric. Moreover, only two families of Gauss cubature 
rules are known to exist for integrals of two variables \cite{LSX, SX94}. In both cases, the corresponding orthogonal 
polynomials are first studied by Koornwinder in his classical papers \cite{K74a} and \cite{K74b} based on symmetric 
functions. 

For the centrally symmetric case, the nonexistence of the Gauss cubature rules can be seen from a lower bound for 
the number of nodes stronger than that of \eqref{eq:old_lwbd}. The bound is called M\"oller's lower bound \cite{Moller} 
and it states that if $W$ and $\Omega$ are centrally symmetric, then the number of nodes $N$ of a cubature rule of 
degree $2n-1$ for $\CI(f)$ satisfies 
\begin{equation} \label{eq:Moller}
   N \ge  \binom{n+1}{2} + \left  \lfloor \frac n 2 \right \rfloor. 
\end{equation}
The first family \footnote{By a family, we mean a family of cubature rules of degree $n$ for all $n$ in $\NN$ or for $n$
in an infinite subset of $\NN$.} of cubature rules that satisfy M\"oller's lower bound, which is then minimal, is given in 
\cite{MP} for the product Chebyshev weight function on the square $[-1,1]^2$ for even $n$. This family has been 
completed and extended to other product Chebyshev weight functions and derived via several methods; see 
\cite{BP, CoolS, LSX09, Sch, X94}. These remain, however, only known families of minimal cubature rules until a large 
collection of weight functions that admit minimal cubature rules on $[-1,1]^2$ is identified in \cite{X12a}. This latter 
collection has its roots in the families of Koornwinder polynomials in \cite{K74a}, which will be described below. 

The paper is organized as follows. In the next section, we describe Koornwinder polynomials and Gauss cubature 
rules. Minimal cubature rules are discussed in the third section, where the examples are given via an involution of
integrals on the square, which allows us to tap into the properties of Koornwinder polynomials. Further minimal cubature
rules are given in the fourth section. 

\section{Koornwinder polynomials and Gauss cubature rules}
\setcounter{equation}{0}

Let $W$ be a nonnegative weight function on a compact domain $\Omega$ of $\RR^2$. We consider orthogonal polynomials 
with respect to the inner product 
$$
  \la f, g\ra_W = \int_\Omega f(x_1,x_2) W(x_1,x_2) \d x_1 \d x_2. 
$$
Let $\CV_n^2$ be the space of orthogonal polynomials of degree $n$. Then $\dim \CV_n^2 = n+1$ as shown by applying 
the Gram-Schmidt process on $\{x_1^n, x_1^{n-1}x_2, \ldots, x_1 x_2^{n-1}, x_2^n\}$. We refer to properties of orthogonal
polynomials in two or more variables to \cite{DX}. 

We first recall the definition of the Koornwinder polynomials. There are two families of these polynomials, 
both coming from symmetric polynomials. We shall consider only one family, first studied in \cite{K74a}. 

\subsection{Koornwinder polynomials}
Let $w$ be a nonnegative weight function on $[-1,1]$. The mapping $\x = (x_1,x_2) \mapsto \u = (u_1,u_2)$, given by 
\begin{equation} \label{eq:x-u}
u_1 = x_1+x_2 \quad \hbox{and}\quad u_2 = x_1 x_2,
\end{equation}
maps the triangle $\triangle = \{(x_1,x_2): x_1 \le x_2\}$, which is half of $[-1,1]^2$, one-to-one and onto the domain 
$\Omega$ give by
\begin{equation}\label{eq:curved}
  \Omega = \left \{(u_1,u_2): u_1^2 \ge 4 u_2, 1 + u_2 \ge |u_1|, \quad -2 \le u_1 \le 2 \right \},
\end{equation}
which is bounded by two lines and a parabola. For $\g \ge  -\f12$, let $W_\g$ be the weight function defined
on $\Omega$ by 
\begin{equation} \label{eq:Wcurved}
   \CW_\g (u_1,u_2) = w(x_1)w(x_2) (u_1^2 - 4 u_2)^\g, \quad (u_1,u_2) \in \Omega 
\end{equation}
under the mapping \eqref{eq:x-u}. Since the map $\ub\mapsto \xb$ maps $\Omega$ to $\triangle$ that is half of the 
product domain $[-1,1]^2$, we obtain
\begin{equation} \label{eq:SymmIntd=2}
  \int_\Omega f(\u) \CW_\g(\u) \d \u =
     \frac12 \int_{\triangle} f(x_1+x_2,x_1x_2) w(x_1)w(x_2) |x_1-x_2|^{2\g+1}\d \x,
\end{equation}
since the Jacobian of the map is $|x_1-x_2|= \sqrt{u_1^2-4u_2}$. 

Let $\CV_n(\CW_\g)$ denote the space of orthogonal polynomials of degree $n$ with respect to $\CW_\g$ on
the domain $\Omega$. As shown in \cite{K74a}, an orthogonal basis can be given explicitly when 
$\g = \pm \f12$. 

\begin{prop} \label{prop:OPbiangle}
Let $p_n=p_n(w)$ be the orthonormal polynomial of degree $n$ with respect to $w$. Under the mapping \eqref{eq:x-u},
define 
\begin{equation} \label{eq:SymmOP-12}
  P_k^{n,-\f12}(u_1,u_2) = \begin{cases} p_n(x_1) p_k(x_2) + p_n(x_2) p_k(x_1), & k < n, \\
          \sqrt{2} p_n(x_1)p_n(x_2), & k = n. \end{cases}
\end{equation} 
Then $\{P_k^{n,-\f12}: 0 \le k \le n\}$ is an orthonormal basis of $\CV_n(\CW_{-\f12})$. Moreover, under the mapping
\eqref{eq:x-u}, define 
\begin{equation} \label{eq:SymmOP12}
  P_k^{n, \f12}(u_1,u_2) = \frac{p_{n+1} (x_1) p_k(x_2) - p_{n+1} (x_2) p_k(x_1)}{x_1-x_2}, \quad 0 \le k \le n. 
\end{equation} 
Then $\{P_k^{n,\f12}: 0 \le k \le n\}$ is an orthonormal basis of $\CV_n(\CW_{\f12})$. 
\end{prop}

\begin{exam} \label{ex:W-Jac}
If $w(t) = (1-t)^\a(1+t)^\b$ is the Jacobi weight function on $[-1,1]$, then the weight $W_\g$ becomes
\begin{equation*}  
 \CW_{\a,\b,\g}(u_1,u_2) = (1-u_1+u_2)^\a (1+u_1+u_2)^\b (u_1^2- 4 u_2)^\g, \quad \a,\b> -1,
\end{equation*}
on the domain $\Omega$ given in \eqref{eq:curved}, since $(1\pm x_1)(1\pm x_2) = 1 \pm u_1 +u_2$.
\end{exam}

The above definition is also valid if the support of $w$ is unbounded. 

\subsection{Gauss cubature rules}

Let $\PP_n =  \{P_1^n, \ldots, P_n^n\}$ be an orthonormal basis for the space $\CV_n^2$. A zero of $\PP_n$ is a 
point $\x \in \RR^2$ that is a zero of every element in $\CV_n^2$. The zeros of $\PP_n$ characterize the Gauss 
cubature rule of degree $2n-1$ \cite{My}. 

\begin{thm} \label{thm:GaussCuba}
A Gauss cubature exists if, and only if, $\PP_n$ has $\dim \Pi_{n-1}^2$ real, distinct zeros. 
\end{thm}

If we regard $\PP_n$ as a column vector, then the orthogonal polynomials satisfy a three-term relation in a vector-matrix form
\cite[Chapter 3]{DX}
$$
   x_i \PP_n = A_{n,i} \PP_{n+1} + B_{n,i} \PP_n + A_{n-1,i}^\Tr \PP_{n-1}, \qquad i = 1,2,
$$
where $A_{n,i}: (n+1) \times (n+2)$ and $B_{n,i}: (n+1) \times (n+1)$ are called structural matrices. If $W$ and $\Omega$ 
are centrally symmetric, then $B_{n,i} = 0$. The structural matrices can be used to study the zeros of $\PP_n$ \cite{X94b}.

\begin{thm} \label{thm:GaussCuba2}
Let $\PP_n$ be a basis of $\CV_n^2$. Then $\PP_n$ has $\dim \Pi_{n-1}^2$ real zeros if and only if 
\begin{equation} \label{eq:AA=AA}
  A_{n-1,1} A_{n-1,2}^\Tr = A_{n-1,2} A_{n-1,1}^\Tr. 
\end{equation}
In particular, the Gauss cubature rule for $W$ exists if and only if \eqref{eq:AA=AA} holds. 
\end{thm}

Since matrices do not usually commute, the Gauss cubature rules exist rarely. The first families of such rules are given in 
\cite{SX94} based on the zeros of Koornwinder polynomials. Let $P_k^{n, \pm\f12}$ be the orthogonal polynomials 
for $W_{\pm \f12}$ given in Proposition \ref{prop:OPbiangle}. We write
$$
\PP_n^{\pm \f12} = \left(P_0^{n, \pm \f12}, \ldots P_n^{n, \pm \f12}\right). 
$$  
Let $t_{k,n}$, $1 \le k\le n$, be the zeros of $p_n(w)$ and define 
$$
   u_{j,k}^{(n)} = t_{j,n} +t_{k,n} \quad \hbox{and} \quad v_{j,k}^{(n)} = t_{j,n} t_{k,n}, \quad 1 \le k,j \le n.
$$

\begin{prop} \label{prop:GaussZero}
 The polynomials in $\PP_n^{-\f12}$ have $\binom{n+1}{2}$ common zeros given by 
\begin{equation} \label{eq:GaussZ-12}
     Z_n^{-\f12} :=  \left \{ \left (u_{j,k}^{(n)}, v_{j,k}^{(n)} \right): 1 \le j \le k \le  n \right\}
 \end{equation}
and the polynomials in $\PP_n^{\f12}$ have $\binom{n+1}{2}$ common zeros given by  
\begin{equation} \label{eq:GaussZ+12}
   Z_n^{\f12} := \left \{ \left (u_{j,k}^{(n+1)}, v_{j,k}^{(n+1)} \right): 1 \le j \le k-1,\,\, k \le n \right\}.
\end{equation}
\end{prop}
 
This can be deduced from the explicit formulas of $P_k^{n,\pm \f12}$ given in \eqref{eq:SymmOP-12} and
\eqref{eq:SymmOP12}. For example, in the case of $P_k^{n,-\f12}$, we see that every polynomial in the
right-hand side of \eqref{eq:SymmOP-12} contains a $p_n(w;x_1)$ or $p_n(w_2)$, so that it vanishes on
$(t_{j,n}, t_{k,n})$ and, consequently, the left-hand side vanishes on $(u_{j,k}^{(n)}, v_{j,k}^{(n)})$ under \eqref{eq:x-u}.
By Theorem \ref{thm:GaussCuba}, this proposition shows that the Gauss cubature 
rules exist for $W_{\pm \f12}$ on the domain $\Omega$. In \cite{SX94}, the condition \eqref{eq:AA=AA} is 
verified. The Gauss cubature rules can be given explicitly. Let the Gauss quadrature rule for $w$ be given by
$$
   \int_\RR f(t) w(t) \d t = \sum_{k=1}^n \l_{k,n} f(t_{k,n}), \qquad \forall f\in \Pi_{2n-1}.
$$

\begin{thm} \label{thm:GaussEx1d=2}
For $\CW_{- \frac12}$ on the domain $\Omega$, the Gaussian cubature rule of degree $2n-1$ is given by 
\begin{equation} \label{GaussCuba-}
  \int_{\Omega} f(\u) \CW_{-\frac12}(\u) \d \u = 2 \sum_{k=1}^n \mathop{ {\sum}' }_{j=1}^k  
          \l_{k,n} \l_{j,n} f\left(u_{j,k}^{(n)}, v_{j,k}^{(n)}\right), \quad f \in \Pi_{2n-1}^2, 
\end{equation}
where ${\sum}'$ means that the term for $j = k$ is divided by 2. For $\CW_{\frac12}$ on the domain
$\Omega$, the Gaussian cubature rule of degree $2n-1$ is given by 
\begin{equation} \label{GaussCuba+}
   \int_{\Omega} f(\u) \CW_{\frac12}(\u) \d \u = 2 \sum_{k=2}^{n+1} \sum_{j=1}^{k-1}  
          \l_{j,k}^{(n+1)} f\left(u_{j,k}^{(n+1)}, v_{j,k}^{(n+1)}\right), \quad f \in \Pi_{2n-1}^2, 
\end{equation} 
where $\lambda_{j,k}^{(n+1)} = \l_{j,n+1} \l_{k,n+1} (t_{j,n+1} - t_{k,n+1})^2$. 
\end{thm}

In the case that $w = w_{\a,\b}$ is the Jacobi weight, the weight function $W_{\pm \f12}  = W_{\a,\b,\pm \f12}$ 
is given in Example \ref{ex:W-Jac}. In particular, for $\a = \b = -\f12$, the weight function is 
$$
   W_{-\f12, -\f12,-\f12} (u_1,u_2) = \left( (1+u_2)^2-u_1^2 \right)^{-\f12} (u_1^2-u_2)^{-\f12}. 
$$
 The nodes of the Gaussian cubature rule of degree $2n-1$ with $n = 20$ are depicted in Figure \ref{fig:parabola_nodes}.   
\begin{figure}[ht]  
\centering
\includegraphics[width=0.8\textwidth]{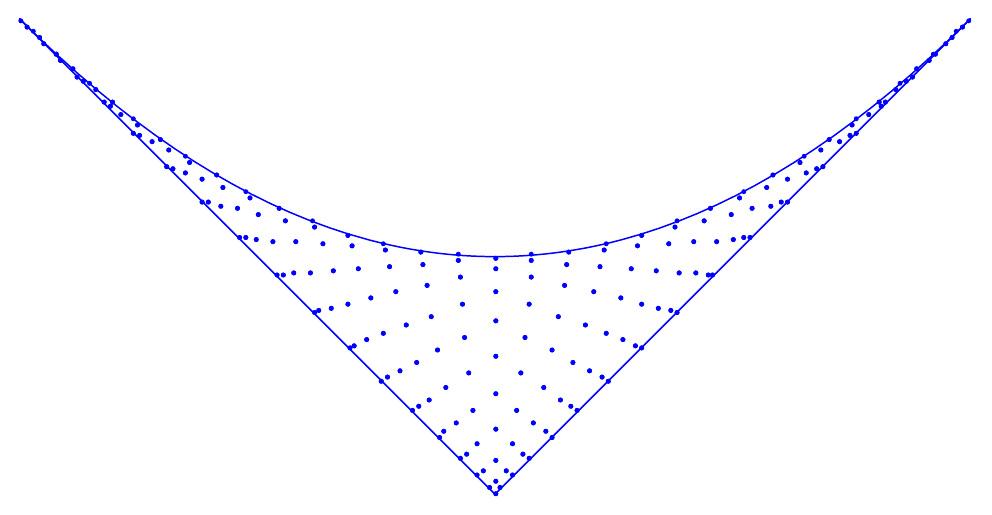}   
\caption{The nodes of Guass cubature of degree 39 for $W_{-\f12,-\f12,-\f12}$} \label{fig:parabola_nodes}
\end{figure}

It should be mentioned that this family of Gauss cubature rules is extended to higher dimensions in \cite{BSX}, 
which leads to the first family of Gaussian cubature rules in $\RR^d$. 

\section{Minimal cubature rules and orthogonal polynomials}
\setcounter{equation}{0}

Recall that a cubature rule for the integral $\CI f$ is minimal if it has the smallest number of nodes among all cubature rules 
of the same degree for $\CI f$. We discuss minimal cubature rules and give a family of examples on the square, which are 
associated with orthogonal polynomials originating from the Koornwinder polynomials

\subsection{Minimal cubature rules and M\"oller's lower boound}
 We need to know a lower bound for the number of nodes to verify that a cubature rule is 
minimal. If it exists, the Gauss cubature rule is minimal since its number of nodes attains the lower bound \eqref{eq:old_lwbd}. 

An important result by M\"oller \cite{Moller} gave an improved lower bound for the number of nodes of a cubature rule 
of degree $2n-1$, which can be made explicit in the centrally symmetric setting. The lower bound is formulated in 
\cite{X94} via the structural matrices of orthogonal polynomials and the result is given as follows.  

\begin{thm}\label{thm:Moller-lwbd}
The number of nodes of a cubature rule of degree $2n-1$ satisfies 
\begin{equation} \label{eq:Moller-lwbd2}
    N \ge \dim \Pi_{n-1}^2 +  \frac{1}{2} \rank \left [A_{n-1,1} A_{n-1,2}^\Tr - A_{n-1,2} A_{n-1,1}^\Tr \right]. 
\end{equation} 
In particular, if $W$ and $\Omega$ are centrally symmetric, then 
\begin{equation} \label{eq:Moller2}
 N \ge \dim \Pi_{n-1}^2 + \left \lfloor \frac n 2 \right \rfloor = \frac12 n(n+1) + \left \lfloor \frac n 2 \right \rfloor. 
\end{equation} 
\end{thm} 

The lower bound \eqref{eq:Moller-lwbd2} can also be extended to several variables. There is another lower bound
that extends \eqref{eq:Moller2} to $d > 2$ in the centrally symmetric setting in $\RR^d$, also due to M\"oller but based 
on a different approach. However, the most interesting case remains in two variables because 
we do not know any family of cubature rules that attain M\"oller's lower bound for $d > 2$, whereas a large family 
of examples is known for $d =2$ on the square. To describe the result, we state a characterization of the cubature 
rules that attain M\"oller's lower bound in \cite{Moller2} when $d =2$. 

\begin{thm}\label{thm:min_cubature}
Let $W$ and $\Omega$ be centrally symmetric. A cubature rule of degree $2 n-1$ attains M\"oller's lower 
bound \eqref{eq:Moller2} if, and only if, its nodes are common zeros of $n + 1 - \lfloor \frac n 2 \rfloor$ 
linearly independent orthogonal polynomials of degree $n$. 
\end{thm}

A characterization of when there are $n + 1 - \lfloor \frac n 2 \rfloor$ linearly independent orthogonal 
polynomials of degree $n$ that have $N_{\mathrm{min}}(n)$ real nodes is given in \cite{X94} in terms of
a non-linear system of equations given via the structural matrices of orthogonal polynomials. 

Based on the characterization in Theorem \ref{thm:min_cubature}, the first family of cubature rules of degree 
$2n-1$ that attains the lower bound \eqref{eq:Moller-lwbd2} appears in \cite{MP}, soon after \cite{Moller}, for 
the product Chebyshev weight on $[-1,1]^2$. Further examples centered around the product Chebyshev weight 
functions have appeared in \cite{BP, CoolS, Sch, X94}, some based on different methods. However, the next group 
of examples was obtained only recently in \cite{X12a}, where a large collection of weight functions is initiated from 
the orthogonality of the Koornwinder polynomials and proved to admit minimal cubature rules. 

\subsection{A family of orthogonal polynomials on the square}
Our examples of minimal cubature rules are established for the weight functions defined below. We consider
the relevant orthogonal polynomials in this subsection. 

\begin{defn} \label{defn:Wgamma}
Let $w$ be a weight function defined on $[-1,1]$. For $\g > -1$, we define $W_\g$ on $[-1,1]^2$ by 
$$
  W_\g(\x) = w(\cos (\t_1-\t_2)) w(\cos(\t_1+\t_2)) \left(1-x_1^2\right)^\g \left(1-x_2^2\right)^\g \left|x_1^2-x_2^2\right|,
$$
where $\x = (x_1, x_2)= (\cos \t_1, \cos \t_2)$ with $0 \le \t_1,\t_2 \le \pi$. 
\end{defn}

These weight functions are closely related to the weight functions $\CW_\g$ defined on the curved domain $\Omega$ 
in \eqref{eq:Wcurved}. As it is shown in Theorem \ref{thm:GaussEx1d=2}, the weight functions $\CW_{\pm \f12}$ admit 
the Gaussian cubature rules of degree $2n-1$ on the domain $\Omega$. The following observation was made in
\cite{X12b}, which connects the two weigh functions. 

\begin{lem} \label{Int-Para-cube}
If $w$ is supported on $[-1,1]$, then 
\begin{equation} \label{eq:CWgamma}
     W_\g (\x) := 4^\g \CW_\g \left(2 x_1 x_2, x_1^2+x_2^2 -1\right)\left |x_1^2-x_2^2\right|, 
\end{equation}
so that $W_\g$ is centrally symmetric. Moreover, for integrable function $f$ on $\Omega$, 
\begin{align} \label{eq:Int-P-Q}
 \int_{[-1,1]^2} f(\ub) \CW_{\g} (\ub) \d \ub 
       = 4 \int_{\fs(w)^2} f\left(2x_1 x_2, x_1^2+x_2^2 -1\right) W_{\g}(\x) \d \x. 
\end{align}
\end{lem}

\begin{exam} \label{eq:Jac-2d}
If $w$ is the Jacobi weight function $w(t) = w_{\a,\b}(t) = (1-t)^\a(1+t)^\b$, $\a,\b > -1$, then, using the identity 
$(1\pm \cos (\t_1-\t_2))(1 \pm \cos (\t_1+\t_2)) = (x_1\pm x_2)^2$,  the weight function $W_\g$ becomes 
$W_{\a,\b,\g}$ defined by
\begin{equation*}
    W_{\a,\b,\g}(\x) =|x_1-x_2|^{2\a+1} |x_1+x_2|^{2\b+1} \left(1-x_1^2\right)^{\g} \left(1-x_2^2\right)^{\g} 
\end{equation*}
for $\x \in [-1,1]^2$, where $\a,\b > -1$ and $\g \ge -\f12$. Correspondingly the weight $\CW_\g$ 
becomes the weight function $\CW_{\a,\b,\g}$ in \eqref{eq:CWgamma}.
\end{exam}

The relation \eqref{eq:Int-P-Q} allows us to derive orthogonal polynomials for $W_\g$ through the Koornwinder polynomials. 
Using the integral relation \eqref{eq:Int-P-Q}, the orthogonal polynomials associated with $W_\g$ on $[-1,1]^2$ can be 
deduced in terms of those associated with $\CW_\g$ and three other related weight functions on the domain $\Omega$
\cite{X12b}. While the weight function $\CW_\g$ in \eqref{eq:Wcurved} is defined via $w$, the other weight functions are 
defined via $(1-x^2)w(x)$, $(1-x) w(x)$, and $(1+x)w(x)$, respectively. More precisely, they are defined by 
\begin{align}\label{Wij}
\begin{split}
  \CW_\g^{(1,1)}(\u) &  :=   (1-u_1+u_2)(1+u_1+u_2)\CW_\g(\u), \\
  \CW_\g^{(1,0)}(\u)  & :=   (1-u_1+u_2)\CW_\g(\u), \\
  \CW_\g^{(0,1)}(\u)  & :=   (1+u_1+u_2)\CW_\g(\u), 
\end{split} \qquad \u \in \Omega,
\end{align} 
since $1-u_1+u_2 = (1-x_1)(1-x_2)$ and $1+u_1+u_2 = (1+x_1)(1+x_2)$ under \eqref{eq:x-u}. 
Let $\{P_{k,n}^{(\g)}: 0 \le k \le n\}$ be an orthogonal basis for $\CV_n(\CW_{\g})$ under the inner product
$$
 \la f, g \ra_{\CW_{\g}} = \int_\Omega f(\u) g(\u) \CW_\g (\u) \d \u.
$$%
We further denote by $P_{k,n}^{(\g),1,1}$, $P_{k,n}^{(\g),1,0,}$, $P_{k,n}^{(\g),0,1}$ the 
orthogonal polynomials of degree $n$ with respect to $\la f, g \ra_\CW$ for $\CW = \CW_\g^{(1,1)}$, 
$\CW_\g^{(1,0)}$, $\CW_\g^{(0,1)}$, respectively. 

\begin{prop}\label{prop:iQkn}
An orthogonal basis of $\CV_{2n}(W_{\g})$ is given by 
$$
 \left \{ {}_1Q_{k,2n}^{(\g)}: 0 \le k \le n \right \}\bigcup \left \{ {}_2Q_{k,2n}^{(\g)}: 0 \le k \le n-1 \right \}
$$%
where 
\begin{align} \label{Qeven2}
\begin{split}
         {}_1Q_{k,2n}^{(\g)}(\x):= & P_{k,n}^{(\g)}(2x_1x_2, x_1^2+x_2^2 -1), \qquad 0 \le k \le n, \\
         {}_2Q_{k,2n}^{(\g)}(\x) := & (x_1^2-x_2^2)  P_{k,n-1}^{(\g),1,1}(2x_1x_2, x_1^2+x_2^2 -1),  \quad 0 \le k \le n-1; 
 \end{split}
\end{align}%
and an orthogonal basis of $\CV_{2n+1}(W_{\g})$ is given by 
$$
 \left \{ {}_1Q_{k,2n+1}^{(\g)}: 0 \le k \le n \right \}\bigcup \left \{ {}_2Q_{k,2n+1}^{(\g)}: 0 \le k \le n\right \}
$$%
where 
\begin{align} \label{Qodd2}
 \begin{split}
     &  {}_1Q_{k,2n+1}^{(\g)}(\x):= (x_1+ x_2) P_{k,n}^{(\g),0,1}(2x_1x_2, x_1^2+x_2^2 -1), \qquad  0 \le k \le n,  \\
     &  {}_2Q_{k,2n+1}^{(\g)}(\x):= (x_1-x_2) P_{k,n}^{(\g),1,0}(2x_1x_2, x_1^2+x_2^2 -1), \qquad 0 \le k \le n, 
 \end{split}
\end{align}
\end{prop}

For $\g = \pm \f12$, an explicit basis of $\CV_n(\CW_{\pm \f12})$ can be given in terms of orthogonal polynomials
$p_n(w)$ of one variable, as shown in Proposition \ref{prop:OPbiangle}. Thus, by the relation in Proposition \ref{prop:iQkn}, 
we can state a basis for $W_{\pm \f12}$ in terms of orthogonal polynomials of one variable. 
Let 
$$
w^{i,j}(t) = (1-t)^i (1+t)^j w(t), \qquad (i,j) =(1,1), (1,0), (0,1).
$$

\begin{prop}\label{prop:iQn_st}
Let $w$ be a weight function with support $\s(w) = [-1,1]$. Let $\x = (\cos \t_1, \cos \t_2)$. Then
an orthogonal basis of $\CV_{2n}(W_{-\frac12})$ is given by, for $0 \le k \le n$ 
and $0 \le k \le n-1$, respectively,  
\begin{align*} 
    & {}_1Q_{k,2n}^{(-\frac12)}(\x) = p_n(w;\cos (\t_1 - \t_2)) p_k(w;\cos (\t_1+\t_2)) \\
    &    \qquad\qquad  \qquad\qquad \qquad\qquad  
           +  p_k (w;\cos (\t_1 - \t_2)) p_n (w;\cos (\t_1+\t_2)), \\
    & {}_2Q_{k,2n}^{(-\frac12)}(\x)   = 
       (x_1^2-x_2^2) \left[ p_{n-1} \left(w^{1,1}; \cos (\t_1 - \t_2)\right) p_k \left(w^{1,1} ; \cos (\t_1 + \t_2)\right) \right. \\ 
    &    \qquad\qquad  \qquad\qquad \qquad\qquad   
        \left. +\,  p_k \left(w^{1,1}; \cos (\t_1 - \t_2)\right) p_{n-1}\left(w^{1,1}; \cos (\t_1 + \t_2)\right) \right], 
\end{align*}
and an orthonormal basis of $\CV_{2n+1}(W_{-\frac12})$ is given by, for $0 \le k \le n$, 
 \begin{align*} 
     & {}_1Q_{k,2n+1}^{(-\frac12)}(\x)   = (x_1+x_2) 
           \left[ p_n \left(w^{0,1}; \cos (\t_1 - \t_2)\right) p_k \left(w^{0,1}; \cos (\t_1 +\t_2)\right) \right. \\ 
     &    \qquad\qquad   \qquad\qquad \qquad\qquad   \left. 
           +\, p_k \left(w^{0,1};\cos (\t_1 - \t_2)\right) p_n \left(w^{0,1};\cos (\t_1 + \t_2)\right) \right], \\
     & {}_2Q_{k,2n+1}^{(-\frac12)}(\x)  =  (x_1- x_2) \left[
          p_n \left(w^{1,0}; \cos (\t_1 - \t_2)\right ) p_k \left(w^{1,0};\cos (\t_1 + \t_2)\right)  \right. \\ 
     &    \qquad\qquad  \qquad\qquad   \qquad\qquad 
          \left.    +\,  p_k \left(w^{1,0};\cos (\t_1 - \t_2)\right) p_n \left(w^{1,0};\cos (\t_1 + \t_2)\right) \right].
\end{align*}
\end{prop}

\subsection{Minimal cubature rules of degree $2n-1$}

If $n = 2m$ is even, then the nodes of the cubature rule will be the common zeros of polynomials ${}_1Q_{k,2m}^{(\pm\f12)}$, 
$0\le k \le m$, in $\CV_n(W_{\pm \f12})$. To describe these zeros, let us write the zeros $t_{k,n}$ of the orthogonal 
polynomial $p_n(w)$ as $t_{k,m} =  \cos \t_{k,m}$ with $\t_{k,m} \in (0, \pi)$. Moreover, for $1\le j, k \le m$, we 
denote 
\begin{align} \label{stjk} 
  s_{j,k}^{(m)} =  \cos \tfrac{\t_{j,m}-\t_{k,m}}{2}  &\quad\hbox{and}\quad  t_{j,k}^{(m)} =  \cos \tfrac{\t_{j,m}+\t_{k,m}}{2}.
\end{align}
Furthermore, we denote by $\l_{k,m}(w)$ the weights in the Gauss quadrature rule of degree $2m-1$ for the weight $w$, 
$$
   \int_{-1}^1 f(t) w(t) \d t  = \sum_{k=1}^m  \l_{k,m}(w) f(t_{k,m}), \qquad \deg f \le 2 m-1. 
$$

\begin{thm} \label{thm:cubaCW}
Let $n = 2m$ be an even positive integer.  Then the weight function $W_{\pm \frac12}$, defined in Definition \ref{defn:Wgamma},
admit the minimal cubature 
rule of degree $2n-1 = 4m-1$ with $2m(m+1)$ nodes. More precisely,   
\begin{enumerate}[\rm (i)]
\item for $W_{-\f12}$ and $f\in \Pi_{2n-1}^2$, 
\begin{align} \label{MinimalCuba2-}
  \int_{[-1,1]^2}  f(\x) W_{-\frac12}(\x) \d \x =  \frac12 \sum_{k=1}^m & \mathop{ {\sum}' }_{j=1}^k 
    \l_{j,k}^{(m)} \left[ f\left( s_{j,k}^{(m)}, t_{j,k}^{(m)}\right)+  f\left( t_{j,k}^{(m)}, s_{j,k}^{(m)}\right)    \right .  \notag \\  
        & + \left.  f\left ( - s_{j,k}^{(m)}, - t_{j,k}^{(m)}\right)+  f\left(- t_{j,k}^{(m)},- s_{j,k}^{(m)}\right)  \right],
\end{align}
where $\l_{j,k}^{(m)} = \l_{j,m}(w) \l_{k,m}(w)$ and $\sum'$ means that the terms for $j = k$ are divided by 2.
\item For $W_{\frac12}$ and $f\in \Pi_{2n-1}^2$,  
\begin{align} \label{MinimalCuba2+}
  \int_{[-1,1]^2} f(\x) W_{\frac12}(\x) \d \x  =  \frac12 \sum_{k=2}^{m+1} \sum_{j=1}^{k-1} &
   \mu_{j,k}^{(m+1)} \left[ f\left(s_{j,k}^{(m+1)}, t_{j,k}^{(m+1)}\right)+f\left( t_{j,k}^{(m+1)}, s_{j,k}^{(m+1)}\right)
        \right.\notag  \\  
  & + \left.  f\left(-s_{j,k}^{(m+1)}, -t_{j,k}^{(m+1)}\right)+f\left(-t_{j,k}^{(m+1)},-s_{j,k}^{(m+1)}\right)\right], 
  \end{align} 
where $\mu_{j,k}^{(m+1)} = \l_{j,m}(w) \l_{k,m}(w) (\cos \t_{j,m+1} - \cos \t_{k,m+1})^2$. 
\end{enumerate}
\end{thm}

This is established in \cite{X12a}, where it is shown that these minimal cubature rules can be derived more 
intuitively from several mappings from the Gauss cubature rules in Theorem \ref{thm:GaussEx1d=2}. 

Let $w = w_{\a,\b}$, the the minimal cubature rules of the degree $2n-1 = 4m-1$ are for the weight functions 
\begin{equation} \label{eq:Wab-1/2}
W_{\a,\b,-\f12}(\x) = \frac{|x_1-x_2|^{2\a+1} |x_1+x_2|^{2\b+1}}{\sqrt{1-x_1^2}\sqrt{1-x_2^2}}, \quad \a,\b >  -1,
\end{equation}
when $\g =  \-\f12$. Their $2 m(m+2)$ nodes are given in \eqref{MinimalCuba2-} with $t_{j,n}(w) = t_{j,n}^{(\a,\b)}$
being the nodes of the Jacobi polynomial $P_n^{(\a,\b)}$. When $\a = \b = -\f12$, these are minimal cubature rules 
for the Chebyshev weight functions that first appeared in \cite{MP}. For $\a > -\f12$ and/or $\b > - \f12$, the nodes are 
propelled away from the diagonal(s) of the square $[-1,1]^2$, as shown in Figure \ref{Fig:min47-ab2} for the weigh 
functions 
$$
  W_{\f12,\f12,-\f12}(\x) = \frac{|x_1^2 - x_2^2|}{\sqrt{1-x_1^2} \sqrt{1-x_2^2}} \quad \hbox{and}\quad
  W_{\f12, - \f12,-\f12}(\x) = \frac{|x_1 - x_2|}{\sqrt{1-x_1^2} \sqrt{1-x_2^2}}. 
$$%
\begin{figure}[ht]
\centering
\includegraphics[width=0.45\textwidth]{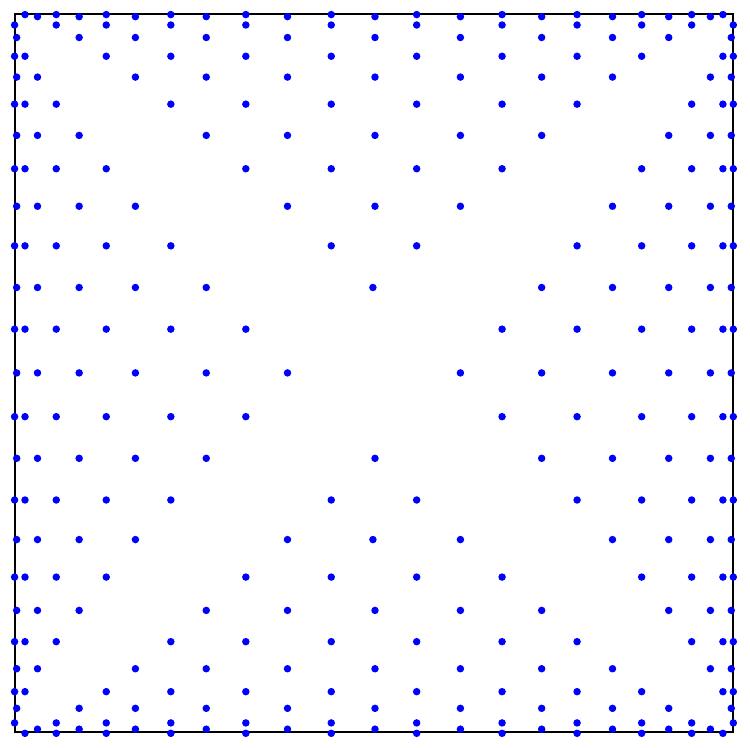}   \quad
\includegraphics[width=0.45\textwidth]{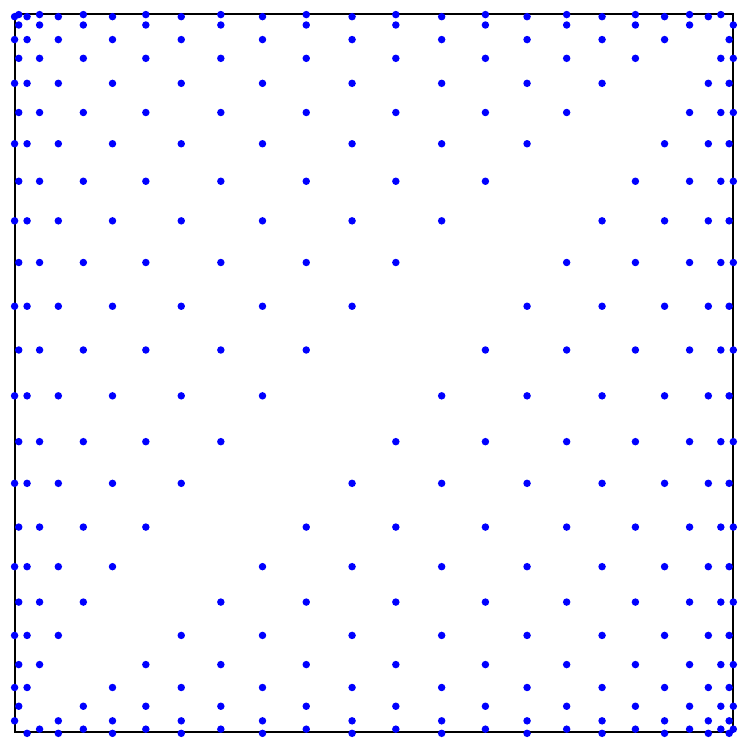} 
\caption{312 nodes of minimal cubature rule of degree $47$ for $W_{\f12, \f12,-\f12}$ (left) and
$W_{\f12, - \f12,-\f12}$ (right)}  \label{Fig:min47-ab2}
\end{figure}

In the case of $n$ is odd,  $n = 2m+1$, the cubature rule of degree $2n-1 =  4 m +1$ that attains the lower bound
\eqref{eq:Moller2} has $N =2(m + 1)^2-1$. By Theorem \ref{thm:min_cubature}, the nodes are zeros of $m+2$ orthogonal 
polynomials of degree $n$. We again expect that the off-diagonal nodes of the minimal cubature rule are common zeros of 
$\{{}_2Q_{k,2m+1}^{\a,\b,\pm \f12}: 0 \le k \le m\}$, but this collection of polynomials consists of $m+1$ polynomials and 
it has an infinite many zeros since they vanish on the $x_2=x_1$ diagonal. Thus, we need one more orthogonal polynomial 
of degree $n= 2m+1$ that vanishes on the off-diagonal nodes and also on $2m$ nodes on the diagonal. For this weight 
function, we need to know the zeros of a particular quasi-orthogonal polynomials, which we shall state only when
$w$ is the Jacobi weight function. 

Let $t_{k,m} = t_{k,m}^{\a+1,\b}$ be the zeros of $P_m^{(\a+1,\b)}$. We adopt the notation in \eqref{stjk}. 
Furthermore, we define 
\begin{align*}
     S_m^{(\pm \f12)} (t) & = P_m^{(\a,\b+1)}(1)P_m^{(\a+1,\b)}(2t^2-1) \mp P_m^{(\a,\b+1)}(2t^2-1)P_m^{(\a+1,\b)}(1).
\end{align*}
It is shown in \cite{X17} that this quasi-orthogonal polynomial has $2m$ distinct zeros in $(-1,1)$. 

\begin{thm} \label{thm:cuba_minO}
Let $n = 2 m+1$. Then the minimal cubature rules of degree $2n-1$ exist for $W_{\a,\b,\pm\f12}$. 
More precisely, 
\begin{enumerate} [\rm (i)]
\item Let $\xi_{k,2m+1} \in (-1,1)$, $1 \le k \le 2m+1$, be the zeros of $t S_m^{(-\f12)}(2t^2-1)$. Then a
minimal cubature rule of degree $2n-1$ for $W_{\a,\b, -\f12}$ has the nodes 
\begin{align*}
 X_{n}^{(-\f12)} =& \left \{ \left(s_{j,k}^{(m)}, t_{j,k}^{(m)}\right), \left( t_{j,k}^{(m)}, s_{j,k}^{(m)}\right),  
        \left( - s_{j,k}^{(m)}, - t_{j,k}^{(m)}\right), \left(- t_{j,k}^{(m)},- s_{j,k}^{(m)}\right):  \right. \\
        & \left.  \quad 1 \le j\le k \le m\right\}  \bigcup \left\{(\xi_{k,2m+1}, \xi_{k,2m+1}): 1 \le k \le 2m+1 \right\}.
\end{align*}
\item Let $\eta_{k,2m+1} \in (-1,1)$, $1 \le k \le 2m+1$, be the zeros of $S_{m+1}^{(\f12)}(2t^2-1)$ inside
$(-1,1)$. A minimal cubature rule of degree $2n-1$ for $W_{\a,\b,\f12}$ has the nodes 
\begin{align*}
 X_{n}^{(\f12)} =& \left \{ \left(s_{j,k}^{(m+1)}, t_{j,k}^{(m+1)}\right), \left( t_{j,k}^{(m+1)}, s_{j,k}^{(m+1)}\right),  
        \left( - s_{j,k}^{(m+1)}, - t_{j,k}^{(m+1)}\right), \left(- t_{j,k}^{(m+1)},- s_{j,k}^{(m+1)}\right):  \right. \\
        & \left.  \quad 0 \le j <  k\le m+1\right\}  \bigcup \left\{(\eta_{k,2m+1}, -\eta_{k,2m+1}): 1 \le k \le 2m+1 \right\}.
\end{align*}%
\end{enumerate}
\end{thm}

This family of minimal cubature rules has the same type of behavior, as when $n$ is even, that the nodes are being 
propelled away from the diagonal when $\a$ and $\b$ increase, but there are always $2m$ points on the diagonal
$x_1 = x_2$. 

\section{Further minimal cubature rules of $2n-1$ with $n$ even} \label{sec:futher_min}
\setcounter{equation}{0}

Let $w$ be a weight function defined on $[-1,1]$. We consider minimal cubature rules for a more general class 
of weight functions defined by, for $\ell = 0, 1, 2,\ldots$, and $x_1 = \cos \t_1$ and $x_2 = \cos \t_2$, 
$$
 W_{-\frac12}^{(\ell)}(\x) 
     =  w\left(\cos \ell (\t_1-\t_2)\right) w\left(\cos \ell (\t_1+\t_2)\right) 
        \frac{\left| T_{\ell}(x_1)^2- T_{\ell}(x_2)^2\right|}{\sqrt{1-x_1^2} \sqrt{1-x_2^2} },
$$%
where $T_\ell$ is the Chebyshev polynomial, which agrees with $W_{-\f12}$ when $\ell = 0$. For $\ell \ge 1$, this 
weight function agrees with $W_{-\f12}$ in (i) of Definition \ref{defn:Wgamma} by replacing $w$ 
by $w^{(\ell)}$ defined by 
\begin{equation}\label{eq:w^ell}
    w^{(\ell)} (t) = w\left(T_\ell(t) \right) \frac{\sqrt{1- T_\ell^2(t)}} {\sqrt{1-t^2}}, \qquad -1 \le t \le 1.  
\end{equation}
Consequently, by Theorem \ref{thm:cubaCW}, the minimal cubature rules of degree $2n-1$ with $n$ even 
exist, and the nodes of the minimal cubature rule for $W_{-\f12}^{(\ell)}$ are based on zeros of $p_{n/2}(w^{(\ell)})$. 
For $n = 2 \ell m$, we provide an explicit constructive below so that the nodes are given in terms of the zeros 
$t_{k,m}$ of $p_n(w)$ and the wights $\l_{k,m}(w)$ are those of the Gaussian quadrature associated to $w$. 

Let $G_\ell$ be the group defined by rotations of multiples of $2\pi/ \ell$ about the origin. Denote the elements 
of $G_\ell$ by the angle of rotations, $G_\ell = \{\tau_{k,\ell} = 2 \pi k/ \ell: 0 \le k \le \ell -1\}$ and define the action
of $G_\ell$ on $f: [-1,1] \mapsto \RR$ by $\tau_{k,\ell} f(t) = f(\cos (\t+ \tau_{k,\ell}))$. 

\begin{lem}
Let $\ell = 1,2,\ldots$. Then the orthogonal polynomial of degree $ n = \ell m$ associated with $w^{(\ell)}$ is
$$
   p_{\ell m}\left(w^{(\ell)}; t\right ) = p_m\left (w; T_\ell(t) \right), \qquad m =0,1,2,\ldots. 
$$
\end{lem}

\begin{proof}
We start with an elementary integral identity 
\begin{align} \label{eq:intTell=T}
\int_{-1}^1 f\left(T_\ell( t) \right) \frac{\d t} {\sqrt{1-t^2} }\d t = \int_{-1}^1 f(t) \frac{\d t}{\sqrt{1-t^2}} \d t. 
\end{align}%
This follows from a changing variable $t \mapsto \cos \t$ so that 
\begin{align*}
\int_{-1}^1 f\left(T_\ell( t) \right) \frac{\d t} {\sqrt{1-t^2} }\d t \, & = \int_0^\pi f(\cos \ell \t) \d \t
 =  \frac{1}{\ell}\int_0^{\ell \pi} f(\cos \t) \d \t \\
  & = \frac{1}{\ell} \sum_{k=0}^{\ell-1} \int_{k}^{k+1} f(\cos \t) \d \t = \int_0^\pi f(\cos \t) \d \t,
\end{align*}
where the last step follows a change of variable $\t \mapsto k \pi + \phi$ if $k$ is even and $\t \mapsto (k+1) \pi - \phi$ 
if $k$ is odd. For $k =1,2,\ldots$, we claim that 
$$
   \int_{-1}^1 f\left(T_\ell (t) \right) T_k (t) \frac{\d t}{\sqrt{1-t^2}} = 0, \qquad \hbox{if $ k \nequiv 0 \mod \ell$}. 
$$%
Indeed, changing variable $t \mapsto \cos \t$, the integral in the left-hand side becomes
\begin{align*}
   \int_0^\pi f(\cos \ell \t) \cos k \t \d \t & = \frac12 \int_{-\pi}^\pi f(\cos \ell t) \cos k \t \d \t \\
    &  = \frac1{2 \ell} \sum_{j=0}^{\ell-1}  \int_{-\pi}^\pi f(\cos \ell t) \cos k \left(\t + \tfrac{2\pi j}{\ell}\right) \d \t
\end{align*}
by the $2\pi$-periodicity. Writing $\cos k (\t + \tau_j) = \cos k \t \cos k \tau_j  - \sin k \t \sin k \tau_j$, the last
integral is zero because 
$$
    \sum_{j=0}^{\ell -1} \left( \cos k \tfrac{2\pi j}{\ell} + \i \sin \tfrac{2\pi j}{\ell}\right)
        =  \sum_{j=0}^{\ell -1} \e^{\i  \frac{2\pi k j}{\ell}} = 0 \quad \hbox{if $ k \nequiv 0 \mod \ell$}.
$$ 
Let $\varpi(t) = w(t) \sqrt{1- t^2}$ so that $w^{(\ell)}(t) = \varpi(T_\ell(t))/\sqrt{1-t^2}$. Then, by \eqref{eq:intTell=T} and
$T_{\ell j}(t) = T_j(T_\ell(t))$, 
\begin{align*}
  \int_{-1}^1 p_m\left (w; T_\ell(t)\right) T_k(t) w^{(\ell)}(t) \d t = \int_{-1}^1 p_m (w; t) T_j(t) w(t) \d t =0
\end{align*}
if $k = \ell j$ and $0 \le j \le m-1$, whereas the integral on the left-hand side is zero if $k \nequiv 0 \mod \ell$. This
shows that $p_m\left (w; T_\ell(t)\right)$ is an orthogonal polynomial of degree $\ell m$ associated with $w^{(\ell)}$. 
\end{proof}

\begin{lem}
Let $\ell = 1,2,\ldots$. Then the integral 
$$
    \CI^{(\ell)} f: =  \int_{[-1,1]^2} f(\x) W_{-\frac12}^{(\ell)}(\x) \d \x 
$$%
is invariant under the action of the group $G_{\ell} \times G_{\ell}$. Moreover,
\begin{equation} \label{eq:IntTn=Int}
    \int_{[-1,1]^2} f\big(T_\ell (x_1),T_\ell(x_2) \big) W_{-\frac12}^{(\ell)}(\x) \d \x  = \int_{[-1,1]^2} f(\x) W_{-\frac12} (\x) \d \x.
\end{equation}
\end{lem}

\begin{proof}
Changing variables $(x_1,x_2) \mapsto (\cos \t_1, \cos \t_2)$, we obtain 
$$
\CI^{(\ell)} f = \int_0^\pi \int_0^\pi f(\cos \t_1, \cos\t_2) \varpi\left( \ell \t_1, \ell \t_2\right) \d \t_1 \d \t_2 
$$%
where the weight function $\varpi$ is defined by 
$$
   \varpi(\t_1,\t_2) =  w(\cos (\t_1-\t_2)) w(\t_1+\t_2)) \left| \cos^2 \t_1 - \cos^2\t_2)\right|. 
$$%
form which we conclude that $\CI^{(\ell)}$ is invariant, as $\varpi\left( 2\ell \t_1, 2\ell \t_2\right)$ is, under 
$G_{\ell}\times G_{\ell}$. The identity \eqref{eq:IntTn=Int} follows from applying \eqref{eq:intTell=T} on both 
variables. 
\end{proof}
 
Let $\ell =1,2,3,\ldots$. For $s = \cos \t$ and $t = \cos \phi$ with $0 \le s,t < \pi$, we define sets of points that will
be used to define the nodes of the cubature rules. We need to consider two cases, depending on the parity of $\ell$. 

\medskip\noindent
{\bf Case 1. $\ell$ even}. Assume $\ell = 2,4,6,\ldots$, we define 
$$
  X_\ell^-(s,t) = \bigcup \left \{\left(\cos\frac{(2\nu_1+1) \pi \pm \t}{\ell}, \cos\frac{(2\nu_2+1) \pi \pm \phi}{\ell}\right), \,
     0\le \nu_1, \nu_2 \le  \frac{\ell}{2} -1 \right \},
$$%
where the union is over all four sign changes, and define  
\begin{align*}
X_\ell^+(s,t) = \bigcup \left \{ \left(\cos\frac{2\nu_1 \pi \pm \t}{\ell}, \cos\frac{2\nu_2 \pi \pm \phi}{\ell}\right), 
   \delta_{\pm} \le \nu_1, \nu_2 \le \frac{\ell}2 -1 + \delta_{\pm} \right \}
\end{align*}%
where $\delta_+ = 0$ and $\delta_- = 1$. 

\medskip\noindent
{\bf Case 2. $\ell$ odd}. Assume $\ell = 1,3,5, \ldots$, we define 
$$
  X_\ell^-(s,t) = \bigcup \left \{\left(\cos\frac{(2\nu_1+1) \pi \pm \t}{\ell}, \cos\frac{(2\nu_2+1) \pi \pm \phi}{\ell}\right),
     0\le \nu_1, \nu_2 \le  \frac{\ell-1}{2} -1+\delta_{\pm} \right \},
$$%
where $\delta_+ =0$ and $\delta_- = 1$ and define, with the same $\delta_\pm$,   
\begin{align*}
X_\ell^+(s,t) = \bigcup \left \{ \left(\cos\frac{2\nu_1 \pi \pm \t}{\ell}, \cos\frac{2\nu_2 \pi \pm \phi}{\ell}\right),  
  \, \delta_{\pm} \le \nu_1, \nu_2 \le \frac{\ell-1}2 \right \}.
\end{align*}%

The choices of $\delta_\pm$ ensure that the angular argument of the points are all in $[0,\pi)$ so that all points
in $X_\ell^\pm(s,t)$ are in $(-1,1]^2$. These sets are selected to have the property \eqref{eq:Tn(u)=s} below.  

\begin{lem}\label{lem:7.4.2}
For $\ell = 2,3,\ldots$, 
\begin{align}\label{eq:Tn(u)=s}
\begin{split}
   \big(T_{\ell}(u), T_{\ell}(v)\big) \, & = (s,t), \quad\quad\, \hbox{if $(u,v) \in X_\ell^+(s,t)$}, \\
  \big(T_{\ell}(u), T_{\ell}(v)\big) \, & = (-s,-t), \quad \hbox{if $(u,v) \in X_\ell^-(s,t)$}, 
\end{split}
\end{align}%
Moreover, the cardinality of $X_\ell^\pm(s,t)$ is given by 
$$
 \left|X_\ell^\pm(s,t)\right | = \begin{cases}  \ell^2,  &   s \ne 1, \, t \ne 1, \\
       \frac{\ell (\ell+1)}{2} &   s =1,\, , t \ne 1 \, \hbox{or} \, s \ne 1, \, t =1, \,  \hbox{if $\ell$ is odd}.
         \end{cases}
$$%
whereas if $s =1$, $t \ne 1$ or $s \ne 1$, $t =1$, 
\begin{align*}
  \left|X_\ell^-(s,t)\right | \,& = \frac{\ell^2}{2} \quad \hbox{and}\quad \left|X_\ell^+(s,t)\right | = \frac{\ell^2}{2} +\ell \quad \hbox{if $\ell$ is even}\\
   \left| X_\ell^\pm (s,t) \right | \,& = \frac{\ell (\ell+1)}{2}, \qquad  \hbox{if $\ell$ is even}.
\end{align*}
\end{lem}

\begin{proof}
The identity \eqref{eq:Tn(u)=s} is immediate and the reason that $X_\ell^\pm$ is defined. Since $s = 1$ or $t=1$ is the same as 
$\t =0$ or $\phi =0$, the cardinality of $X_\ell^\pm(s,t)$ is straightforwardly $4 \times (\ell/2)^2$ when $\t \ne 0$ and $\phi \ne 0$. 
If $s =0$ and $t=\ne 0$, or $s\ne 0$ and $t=0$, then the index for $\t=0$, say, is between $0 \le (\ell-1)/2$, so that the 
cardinality of $X_\ell^\pm(1,t)$ is $\ell \times (\ell+1)/2$, whereas if $\ell$ is even, then $X_\ell^-(s,t)$ is $2 \times (\ell/2)^2 = \ell^2/2$
and $X_\ell^+(1,t)$, say, contains $\ell/2+1$ distinct $\nu_1$, so that $|X_\ell^+(1,t)| = 2 \times (\ell/2+1) \ell/2 = \ell^2/2+\ell$. 
\end{proof} 
  
Let $w$ be a weight function on $[-1,1]$ and let $t_{k,m} = \cos \t_{k,m}$ be the zeros of the polynomials $p_n(w)$. We
retain the notation \eqref{stjk} but rewrite it further as 
\begin{align*}
     s_{j,k}^{(m)} & = \cos \t_{j,k}^{(m)}  \quad \hbox{with} \quad \t_{j,k}^{(m)} = \frac{\t_{j,m}- \t_{k,m}}{2}, \\
      t_{j,k}^{(m)} & = \cos \phi_{j,k}^{(m)}   \quad \hbox{with} \quad \phi_{j,k}^{(m)} = \frac{\t_{j,m}+ \t_{k,m}}{2}.
\end{align*}%
We further define the set 
 \begin{align*}
   X_{\ell}(j,k):=   X_\ell^-\left(-s_{j,k}^{(m)}, -t_{j,k}^{(m)}\right)
       &   \bigcup X_\ell^-\left(-t_{j,k}^{(m)}, -s_{j,k}^{(m)}\right), \\ 
        &  \bigcup X_\ell^+\left(s_{j,k}^{(m)}, t_{j,k}^{(m)}\right) \bigcup X_\ell^+\left(t_{j,k}^{(m)}, s_{j,k}^{(m)}\right).
\end{align*}
As in Theorem \ref{thm:cubaCW}, let $\l_{k,m}(w)$ denote the quadrature weights of the Gauss quadrature associated
to $w$ on $[-1,1]$. 

\begin{thm}
Let $n = 2\ell m$ for $\ell = 1, 2,\ldots$.  Then the weight function $W_{-\frac12}^{(\ell)}$ admits a minimal cubature 
rule of degree $2n-1 =4 \ell m-1$ with $2 \ell^2 m^2+2\ell m$ nodes. More precisely,   
\begin{align}\label{eq:cuba_invaTn}
 \int_{[-1,1]^2} f(\x) W_{-\f12}^{(\ell)}(\x) \d \x = \frac12 \sum_{k=1}^m \mathop{ {\sum}' }_{j=1}^k \mu_{k,m}  \mu_{j,m}  
  \frac{1}{|X_\ell(j,k)|}\sum_{(s,t) \in X_\ell (j,k)} f(s,t),
\end{align}
where $|X_\ell(j,k)| = 4 \ell^2$ if $j \ne k$ and $|X_\ell(k,k)| = 2\ell^2 + 2\ell$.
\end{thm} 

\begin{proof}
By the periodicity of the cosine function, the right-hand side of \eqref{eq:cuba_invaTn} is invariant under $G_\ell \times G_\ell$. 
Using the identity in \eqref{eq:IntTn=Int} and applying the minimal cubature rule in Theorem \ref{thm:cubaCW}, we obtain
\begin{align*}
 & \int_{[-1,1]^2} f\left(T_{\ell m}(x_1), T_{\ell m}(x_2)\right)W_{-\f12}^{(\ell)}(\x) \d \x 
       = \int_{[-1,1]^2} f(x_1), x_2)W_{-\f12} (\x) \d \x \\
  = &   \frac12 \sum_{k=1}^m \mathop{ {\sum}' }_{j=1}^k \l_{j,k} 
   \left[ f\left( s_{j,k}^{(m)}, t_{j,k}^{(m)}\right)+  f\left( t_{j,k}^{(m)}, s_{j,k}^{(m)}\right) 
                     +  f\left ( - s_{j,k}^{(m)}, - t_{j,k}^{(m)}\right)+  f\left(- t_{j,k}^{(m)},- s_{j,k}^{(m)}\right)  \right] \\
  = & \frac12 \sum_{k=1}^m \mathop{ {\sum}' }_{j=1}^k \l_{j,m}(w) \l_{k,m}(w)
           \frac{1}{|X_\ell(j,k)|}\sum_{(u,v) \in X_\ell (j,k)} f  \big(T_{\ell}(u), T_{\ell}(v)\big),              
\end{align*}
where $\l_{j,k} = \l_{j,m}(w) \l_{k,m}(w)$ and the last equal sign follows from \eqref{eq:Tn(u)=s}, which holds for all
polynomials $f\in \Pi_{4m-1}^d$. This establishes the cubature rule \eqref{eq:cuba_invaTn} for all polynomials of 
the form $f\left(T_{\ell}(x_1), T_{\ell}(x_2)\right)$ of degree $4 \ell m -1$, which are all invariant 
polynomials under $G_\ell \times G_\ell$ of the degree less than $4 \ell m$. By Sobolev's Theorem on invariant
cubature rules \cite{Sobolev}, which assures that we only need to verify invariant polynomials for an invariant cubature rule, 
this proves \eqref{eq:cuba_invaTn} for all polynomials of degree $4 \ell m -1$. Finally, by 
Lemma \ref{lem:7.4.2}, we obtain $|X_\ell(j,k)| = 4 \ell^2$ if $j \ne k$ and $|X_\ell(k,k)| = 2 \ell (\ell+1)$. Since
$s_{k,k} = \cos 0 =1$, the number of nodes of the cubature rule is equal to 
$$
      4 \ell^2 \left(\frac{m(m+1)}{2} - m \right)+ 2 \ell (\ell+1) m = 2 \ell^2 m^2 + 2 \ell m, 
$$%
which agrees with M\"oller's lower bound when $n = 2 \ell m$. 
\end{proof}

For $w$ being the Jacobi weight function $w_{\a,\b}(t) = (1-t)^\a(1+t)^\b$, $\a, \b > -1$, the weight function 
$W_{-\f12}^{(\ell)}$ becomes 
\begin{equation}\label{eq:WabTn}
  W_{\a,\b, -\f12}^{(\ell)}(\x) = \frac{|T_\ell(x_1) - T_\ell(x_2)|^{2 \a+1} |T_\ell(x_1) +T_\ell (x_2)|^{2 \b+1}}
      { \sqrt{1-x_1^2} \sqrt{1-x_2^2}}. 
\end{equation}%
For $\ell =1$, this is the weight function in \eqref{eq:Wab-1/2}. For $\ell = 2$, the weight function becomes
$$
W_{\a,\b,-\f12}^{(\ell)}(\x) = 2^{2\a+2\b+2} \frac{|x_1^2- x_2^2 |^{2 \a+1} |1 - x_1^2-x_2^2|^{2 \b+1}}
         { \sqrt{1-x_1^2} \sqrt{1-x_2^2}}. 
$$%

As an example, we consider the cubature rules of degree $2n-1 = 8m-1$ for $W_{\a,\b,-\f12}^{(2)}$, which 
has $8 m^2+4m$ number of nodes that are derived from the zeros of the Jacobi polynomials $P_m^{(\a,\b)}$. 
For $\beta > 0$, the nodes of the cubature rules are propelled away from the circle $x_1^2+x_2^2 =1$ and 
they are also propelled away from diagonals of the square if $\alpha > 0$ as well, as shown in Figure \ref{fig:minGGen47}.
\begin{figure}[ht]
\centering
\includegraphics[width=0.45\textwidth]{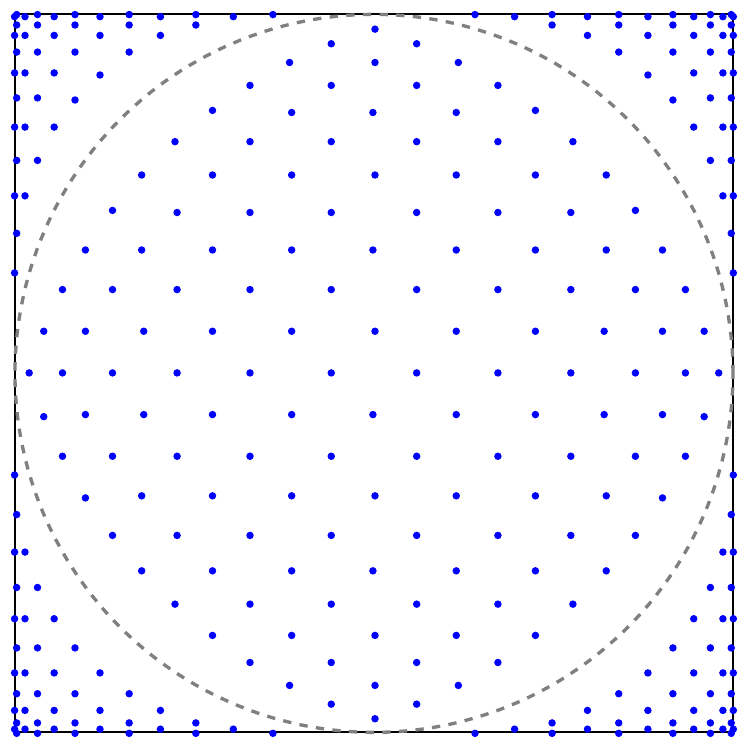}   
 \quad
\includegraphics[width=0.45\textwidth]{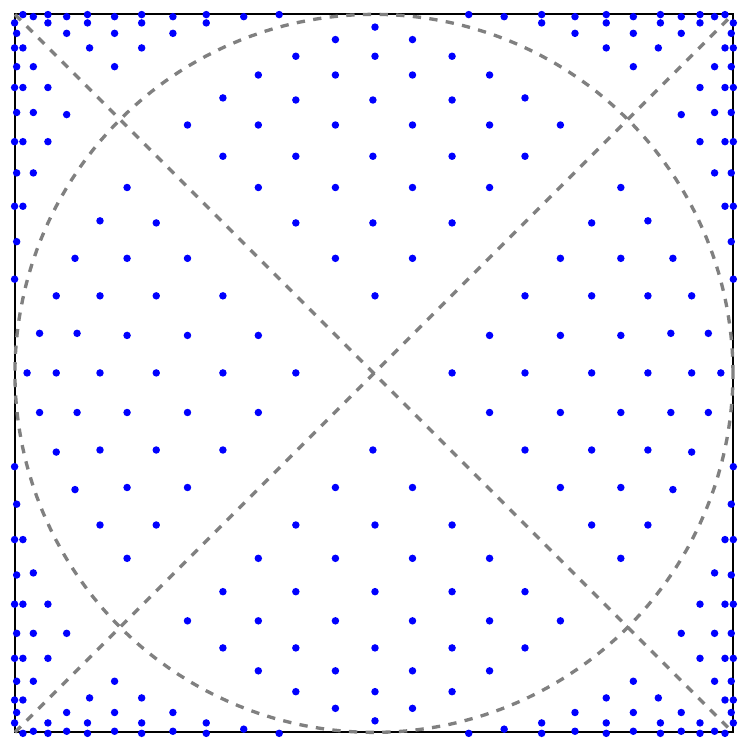} 
\caption{312 nodes of minimal cubature rule of degree $47$ for $W_{\a,\b, -\f12}^{(2)}$ with $\a =\f12$ and $\b=-\f12$ (left) and
$\a =\f12$ and $\b=\f12$ (right)}  \label{fig:minGGen47}
\end{figure}
The nodes will be further away from the unit circle as $\a$ increases and from the diagonals if $\b$ increases.

\end{document}